\documentclass[a4paper]{article}

\usepackage{booktabs}
\usepackage[english]{babel}
\usepackage[utf8]{inputenc}
\usepackage{amssymb}
\usepackage[fleqn]{amsmath}
\usepackage{arydshln}
\usepackage{amsthm}
\usepackage{graphicx}
\usepackage{framed}
\usepackage[colorinlistoftodos]{todonotes}
\usepackage{subcaption}
\usepackage{soul}
\usepackage{multirow}
\usepackage{enumitem}
\usepackage{titling} 
\usepackage{mathtools}
\usepackage{titlesec}
\usepackage[top=1.25in, bottom=1.25in, left=1.25in, right=1.25in]{geometry}
\usepackage{breqn}
\usepackage[numbers]{natbib}
\usepackage{tikz}
\usepackage{pgfplots}
\usepackage{ulem}
\usepackage{caption}
\usepackage{stackengine}
\usetikzlibrary{arrows,%
	petri,%
	topaths}%
\usepackage{algorithm}
\usepackage{algpseudocode}
\usepackage{tabstackengine}

\usepackage{comment}

\TABbinary
\newcounter{example}
\newenvironment{example}[2]{
\refstepcounter{example}
\label{#1}\begin{leftbar}\textbf{\noindent \hspace{-5.3mm}Example \theexample} \\#2\end{leftbar}
}{}

\newcounter{drafts}

\newcounter{module}
\newcounter{model}
\newcommand \linedabstractkw[2]{
  \renewcommand\maketitlehookd{%
    \mbox{}\medskip\par
    \centering
    \hrule\medskip
    \begin{minipage}{0.9\textwidth}
    \textbf{Abstract}\\ #1\\
    
    \textit{Keywords: }#2
    \end{minipage}\medskip\hrule\medskip
    }      
}

\newcommand \ERPWtemplate[3]{
\usepackage{fancyhdr}
\pagestyle{fancy}
\fancyhf{}
\fancyhead[RO]{#1}
\fancyhead[LO]{#2}
\fancyfoot[CO]{#3}
\fancyfoot[RO]{\thepage}

\renewcommand{\headrulewidth}{1pt}
\renewcommand{\footrulewidth}{1pt}
}

\stackMath

\newcommand{\im}[1]{\text{Im}\left(#1\right)}
\newcommand{\re}[1]{\text{Re}\left(#1\right)}

\newcommand{\T}{\mathbb{T}}

\theoremstyle{plain}
\newtheorem{theorem}{Theorem}[section]

\theoremstyle{plain}
\newtheorem{lemma}[theorem]{Lemma}

\theoremstyle{plain}
\newtheorem{proposition}[theorem]{Proposition}

\theoremstyle{remark}

\theoremstyle{plain}

\theoremstyle{plain}
\newtheorem{corollary}{Corollary}[theorem]

\theoremstyle{definition}
\newtheorem{definition}{Definition}

\theoremstyle{plain}

\ERPWtemplate{P. Wissing, E.R. van Dam}{Symmetric spectra of gain graphs}{\leftmark}

\tikzset{group/.style = {shape=circle,draw,dotted,minimum size=1em}}
\tikzset{vertex/.style = {shape=circle,draw,minimum size=1em}}
\tikzset{arc/.style = {->,> = latex'}}
\tikzset{edge/.style = {-,> = latex'}}
\tikzset{negarc/.style = {->,> = latex',dashed}}
\tikzset{negedge/.style = {-,> = latex',dashed}}
\tikzset{tree/.style = {-,> = latex',line width=.7mm}}

\title{
Symmetry in complex unit gain graphs and their spectra
}
\author{
Pepijn Wissing
, Edwin R. van Dam\\ \small{Department of Econometrics and Operations Research, Tilburg University}
}

\pgfplotsset{compat=1.18} 

\begin{document}

\linedabstractkw{
Complex unit gain graphs may exhibit various kinds of symmetry.
In this work, we explore structural symmetry, spectral symmetry and sign-symmetry in such graphs, and their respective relations to one-another. 
Our main result is a construction that transforms an arbitrary complex unit gain graph into infinitely many switching-distinct ones whose spectral symmetry does not imply sign-symmetry.
This provides a more general answer to the analogue of an existence question that was recently treated in the context of signed graphs. 
}{
Complex unit gain graph, eigenvalues, symmetry
}
\maketitle

\section{Introduction}
Throughout the natural sciences, symmetry might be the single most widely recognized feature that is in some way beautiful, useful, or both. 
It, therefore, occurs in various forms. 
Of particular interest to the authors is symmetry in the eigenvalues of graphs and their diverse generalizations. 
A collection of eigenvalues (also called the \textit{spectrum}) is said to be symmetric if it is invariant under multiplication by $-1$.
Two recent works that consider symmetry in eigenvalues are \cite{greaves2022signed, haemers2021spectral}. 

In addition to spectral symmetry, two more instances of symmetry come up in this work. 
Likely the best known of the two is concerned with the existence of multiple assignments of the same collection of labels to graph vertices, that cannot be distinguished by looking at the adjacency and non-adjacency of pairs of vertices. 
That is, a graph is said to be (structurally) symmetric if it has any non-trivial automorphism. 
Structurally, this (asymptotically rare) property implies that parts of the graph are akin to one-another, as well as identical in relation to their mutual complement. 

Finally, we encounter a phenomenon known as sign-symmetry. 
A particular line of research is concerned with the invariance of signed graphs, which was introduced some time ago by \citet{zaslavsky1982signed}, under negation of their sign functions. 
In particular, if $(G,\sigma)$ is switching isomorphic (see Def. \ref{def: sw iso}) to $(G,-\sigma)$, it is said to be sign-symmetric. 
Its relation to spectral symmetry is interesting: while sign-symmetry implies a symmetric spectrum, the reverse relation fails, in general. 
Sporadic examples of signed graphs that show the latter phenomenon on complete graphs were obtained by \citet{belardo2019open}, as well as various such infinite families  on non-complete graphs have been found by \citet{ghorbani2020sign}.

In the current article, we consider these types of symmetry, and the relation between them, in the complex unit gain graph paradigm.
These objects are, effectively, weighted bidirected graphs with weights on the complex unit circle, such that the weight of every arc is equal to the inverse weight of its converse arc. 
In a recent (re)popularization, the generalizations of various well-studied graph theoretical objects, such as signed graphs and the Hermitian adjacency matrices for directed graphs \cite{guo2017hermitian, liu2015hermitian, mohar2020new}, have seen quite some attention.
Applications of gain graphs are primarily related to quantum state transfer \cite{coutinho2014quantum}, and spectral analysis of these objects has yielded a number of interesting parallels to complex geometry, see e.g. \cite{wissing2022unit}.

While some relations between the aforementioned symmetries are clear, the existence of (an infinite family of) signed graphs which are spectrally symmetric but not sign-symmetric was open until recently, when \citet{ghorbani2020sign} constructed the family $\Gamma_s$ (illustrated in Figure \ref{fig: inf family signed graphs}) and proved that such signed graphs exist. 
While mention is made of other infinite families that have the desired property, they seem to effectively involve adding vertices to the leftmost hexagon in Figure \ref{fig: inf family signed graphs}. 
Thus, this article addresses a natural follow-up question, and asks what happens when the framework is generalized to encompass all complex unit gain graphs, rather than just signed graphs.  
We note that by restricting all gains to $\pm i$, one can obtain a gain graph on any underlying graph that is spectrally symmetric, as was already observed in \cite{guo2017hermitian, liu2015hermitian}. However, these gain graphs are also sign-symmetric.

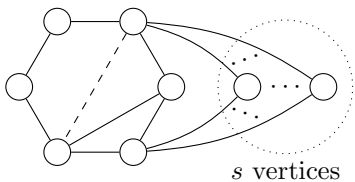
\begin{figure}
    \centering
    \begin{tikzpicture}
            \node[vertex] (6) at (1,0){};
            \node[vertex] (1) at (0.5,0.866){};
            \node[vertex] (2) at (-0.5,0.866){};
            \node[vertex] (3) at (-1,0){};
            \node[vertex] (4) at (-0.5,-0.8660){};
            \node[vertex] (5) at (0.5,-0.866){};
            
            \node[vertex] (7) at (2,0){};
            \node[vertex] (8) at (3,0){};
            
            \node (75) at (2.5,0){$\cdots$};
            \node[rotate=20] (75) at (1.97,.35){$\cdots$};
            \node[rotate=-20] (75) at (1.97,-.35){$\cdots$};
            \node[vertex,dotted] (75) at (2.5,0){~~~~~~~~~~~~~~};
            \node (77) at (2.5,-1.1){$s$ vertices};
            
            \draw[edge] (1) to node{} (6);
            \draw[edge] (1) to node{} (2);
            \draw[edge] (2) to node{} (3);
            \draw[edge] (3) to node{} (4);
            \draw[edge] (4) to node{} (5);
            \draw[edge] (5) to node{} (6);
            \draw[negedge] (1) to node{} (4);
            \draw[edge] (4) to node{} (6);
            
            \draw[edge, bend left=15] (1) to node{} (7);
            \draw[edge, bend left=15] (1) to node{} (8);
            \draw[edge, bend right=15] (5) to node{} (7);
            \draw[edge, bend right=15] (5) to node{} (8);
    \end{tikzpicture}
    \caption{The signed graph $\Gamma_s$. Here, the dashed edge has gain $-1$.}
    \label{fig: inf family signed graphs}
\end{figure}

As a gentle introduction to the subject matter, we first consider spectral symmetry of gain graphs. 
We show that a graph $G$ is underlying to only spectrally symmetric gain graphs if and only if it is bipartite, and that every graph is underlying to some spectrally symmetric gain graphs. 
Then, we consider a number of doubling operations whose origin lies with a well-known recursive construction of Hadamard matrices.
By design, these constructions yield gain graphs with symmetric spectra. 
While most of them also implicitly yield sign-symmetric gain graphs, we prove that a subtle adaptation of Sylvester's double transforms an arbitrary gain graph into infinitely many switching-distinct gain graphs that are not sign-symmetric.

The contents of this paper are organized as follows. 
First, we provide a thorough introduction of all of the concepts used in this article. 
This is followed by a discussion of various constructions of gain graphs with symmetric spectra, in Section \ref{sec: symspec}.
While most of these constructions are sign-symmetric by design, we prove in Section \ref{sec:constructions} that an appropriate adaptation of Sylvester's double, in general, is not.

\section{Preliminaries}

\subsection{Basic definitions}
The objects studied in this work are so-called complex unit gain graphs.
Let $G$ be a bidirected graph with vertex set $V$ and edge set $E$. 
That is, $uv\in E$ if and only if $vu\in E$. 
Let $\mathbb{T}:=\{z\in\mathbb{C}~:~|z|=1\}$ denote the group of complex units. 
Then, the mapping $\psi: E\mapsto \mathbb{T}$ such that $\psi(uv)=(\psi(vu))^{-1}$ is called a \textit{gain function}, and the tuple $\Psi=(G,\psi)$ is called a \textit{complex unit gain graph} \cite{reff2012spectral}. 
For brevity, the words 'complex unit' will usually be omitted. 
$G$ is said to be the \textit{underlying graph} of $\Psi$, sometimes denoted by the operator $\Gamma(\cdot)$. That is, $\Gamma(\Psi)=G$. 

Of essential interest, from a spectral point of view, is the gain of a cycle in $G$. 
Let $C$ be a usual cycle 
in $G$ and let $C^\rightarrow$ denote a directed cycle obtained from $C$ by orienting all edges of $C$ in the same direction. 
That is, every vertex in $C^\rightarrow$ has indegree and outdegree equal to $1$. 
Then the gain of $C^\rightarrow$ is defined as 
\[\phi(C^\rightarrow)=\prod_{e\in E(C^\rightarrow)} \psi(e).\]
In case $C$ is traversed in the reverse direction, say by $C^\leftarrow$, then $\phi(C^\leftarrow)=\overline{\phi(C^\rightarrow)}$. 
Since the traversal direction does not affect the (primarily interesting, see Theorem \ref{thm: coefficients}) real part of the cycle gain, the direction is usually omitted. 

Whenever two cycles $C_1$ and $C_2$ intersect on a single continuous segment consisting of at least an edge, their symmetric difference $C_1\ominus C_2$  is, again, a cycle. 
Moreover, one may compute the gain of the new cycle by taking the product of the gains of the old cycles, making sure that their intersection is traversed in opposite directions.
Loosely put, we have $\phi(C_1\ominus C_2)=\phi(C_1^\rightarrow)\phi(C_2^\leftarrow)$. 
We further note that all cycles are contained in the cycle space, the set of Eulerian subgraphs of a graph, which has a basis of fundamental cycles.

A key role is played by the \textit{gain matrix} $A(\Psi)$ of a gain graph $\Psi$, whose entries are defined by $[A(\Psi)]_{uv}=\psi(uv)$. Note that the gain matrix is Hermitian, and thus diagonalizable with real eigenvalues.
The \textit{characteristic polynomial} $\chi(\lambda)$ of $\Psi$ is said to be the characteristic polynomial of $A(\Psi)$, i.e., 
\begin{equation}\label{eq: charpoly}\chi(\lambda)=\det(\lambda I-A(\Psi))= \lambda^n + a_1\lambda^{n-1} + \ldots + a_n.\end{equation} 
The \textit{eigenvalues} $\lambda_1\geq\lambda_2\geq\ldots\geq \lambda_n$ of $\Psi$ are the roots of $\chi(\lambda)$; the collection of eigenvalues is often referred to as the \textit{spectrum} of $\Psi$.

Two gain graphs $\Psi=(G,\psi)$ and $\Psi'=(G',\psi')$ are said to be \textit{isomorphic} (denoted $\Psi\cong\Psi'$) if there exists a bijective map $\pi: V\mapsto V'$ such that $uv\in E \iff \pi(u)\pi(v)\in E'$ and $\psi(uv)=\psi'(\pi(u)\pi(v))$ for every pair of vertices $u,v\in V$. 
Let $X$ be a diagonal matrix with diagonal elements in $\T$. Then $\Psi'$ is said to be obtained from $\Psi$ by a \textit{diagonal switching} if $A(\Psi') = XA(\Psi)X^{-1}.$
The \textit{converse} of $\Psi$ is obtained by inverting every edge gain of $\Psi$. 
Two gain graphs are considered equivalent when they are switching isomorphic to one-another. 
We adopt the usual definition.
\begin{definition}
\label{def: sw iso}
Two gain graphs are said to be \textit{switching isomorphic} (denoted $\Psi\sim\Psi'$) if one may be obtained from the other by diagonal switching, possibly followed by taking the converse and/or relabeling the vertices. 
Two gain graphs that are not switching isomorphic are said to be \textit{switching-distinct.}
\end{definition}

The main discussion in this article concerns various kinds of symmetry. 
A graph $G$ is said to be symmetric if it has a non-trivial automorphism. 
That is, an isomorphism mapping $\pi$ from $G$ onto itself, preserving adjacency and non-adjacency, such that $\pi(u)\not=u$ for some $u\in V$.  

If a gain graph is such that it has an eigenvalue $\lambda$ with multiplicity $m$ only if it also has an eigenvalue $-\lambda$ with multiplicity $m$, then it is said to have a \textit{symmetric spectrum} or be \textit{spectrally symmetric.}
The following is a well-known property of diagonalizable matrices.
\begin{lemma}
\label{lemma: sym spectra coefficients}
Let $\Psi$ be a unit gain graph with underlying graph $G$ and characteristic polynomial $\chi(\lambda)$ as in \eqref{eq: charpoly}. 
Then $\Psi$ has a symmetric spectrum if and only if $a_j=0$ for all odd $j\leq n$. 
\end{lemma}
We say that a graph $G$ \textit{allows symmetric gain-spectra} if there exists a gain graph $\Psi$ on $G$ whose spectrum is symmetric; if moreover all gain graphs on $G$ have symmetric spectra then we say that $G$ \textit{requires} symmetric gain-spectra. 

The \textit{negation} of a gain graph $\Psi$ is found by multiplying the gain of every arc with $-1$.
This is commonly simply denoted as $-\Psi:=(G,-\psi))$. 
A gain graph is said to be \textit{sign-symmetric} if it is switching isomorphic to its negation; that is, if $\Psi\sim-\Psi$.

Crucially, the gains of cycles are closely related to the spectrum of a gain graph, via the well-known Harary-Sachs coefficients theorem. 
A graph $H$ is called an elementary graph if each of its connected components is either an edge or a cycle.
The characteristic polynomial of a gain graph may be obtained from its elementary subgraphs as follows.
\begin{theorem}[\citet{mehatari2022}]
\label{thm: coefficients}
Let $\Psi$ be a unit gain graph with underlying graph $G$ and characteristic polynomial $\chi(\lambda)$ as in \eqref{eq: charpoly}. 
Then 
\[a_j = \sum_{H\in\mathcal{H}_j(G)}(-1)^{p(H)}2^{c(H)} \prod_{C\in\mathcal{C}(H)}\re{\phi(C)},\]
where $\mathcal{H}_j(G)$ is the set of all elementary subgraphs 
of $G$ with $j$ vertices, $\mathcal{C}(H)$ denote the collection of all cycles in $H$, and $p(H)$ and $c(H)$ are the number of components and the number of cycles in $H$, respectively. 
\end{theorem}

It should be clear that by the above, the gains of the cycles in a cycle basis of a gain graph jointly determine its spectrum. 
Moreover, said cycles are also useful when determining whether or not two gain graphs are switching isomorphic to one-another.
\begin{theorem}[\citet{reff2016oriented}]\label{thm: reff switching}
Let $\Psi=(G,\psi)$ be an order-$n$ gain graph, and let $\mathcal{B}$ be a cycle basis of $G$. 
Then $\Psi'=(\pi(G),\psi')\sim\Psi$ if and only if $\phi(C)=\phi'(\pi(C))$, for all $C\in\mathcal{B}$ and some  automorphism $\pi$ of $G$.
\end{theorem}
Moreover, it immediately follows that we can be sure that two gain graphs are not switching isomorphic when the number of odd $k$-cycles with some gain value $\mu$ does not coincide. In particular, we have the following. 
\begin{corollary}
\label{cor: equal number of cycles sign-sym}
Let $\Psi$ be a gain graph, and let $\gamma_k(\mu)$ denote the number of distinct order-$k$ cycles in $\Psi$ with $\re{\phi(C)}=\mu.$ 
If it holds that $\gamma_{2k-1}(\mu)\not=\gamma_{2k-1}(-\mu)$ for some $\mu\in \mathbb{R}$ and $k\in\mathbb{N}$, then $\Psi$ is not sign-symmetric.
\end{corollary}

\subsection{Initial observations}
We briefly discuss the more obvious relations between the different notions of symmetry. 
Firstly, neither one of spectral symmetry and structural symmetry implies the other. 
For the sake of brevity, we illustrate this with the simplest examples, though there are many to be found. 
Take, for example, any odd cycle $C_{2k+1}$.
This graph clearly has a non-trivial automorphism, but $\Psi=(C_{2k+1},\psi)$ is spectrally symmetric only when $\phi(\Psi)=\pm i$. 
Moreover, taking an arbitrary asymmetric graph $G$ and assigning each of its edges a strictly imaginary gain yields a spectrally symmetric gain graph, disproving the reverse implication. 

Next, it is easy to see that sign-symmetry of $\Psi$ implies the structural symmetry of its underlying graph $\Gamma(\Psi)$, unless every odd cycle has strictly imaginary gain. 
\begin{proposition}\label{prop: sign-symmetric implies nontrivial automorphism}
If the gain graph $\Psi$ is sign-symmetric then $\Gamma(\Psi)$ is symmetric or every odd cycle $C$ in $\Psi$ has $\re{\phi(C)}=0.$
\end{proposition}
\begin{proof}
Suppose that $\Psi$ is switching isomorphic to $-\Psi$, i.e., $A(\Psi) = PX(-A(\Psi)) X^{-1}P^{-1}$ for some diagonal matrix $X$ with $X_{jj}\in \T$ for all $j$ and some permutation matrix $P$. 
Moreover, suppose that at least one odd cycle $C$ has $\re{\phi(C)}\not=0.$ 
Then $\re{\phi(C_\Psi)}\not=\re{\phi(C_{-\Psi})},$ so $P$ is not the identity.
Then $\Gamma(\Psi)$ is not asymmetric since there exists an automorphism of $\Gamma(\Psi)$. Indeed, $\Gamma(\Psi)=P\Gamma(\Psi)P^{-1}.$  
\end{proof}
The reverse implication does not hold. This is easy to see, by taking any symmetric non-bipartite graph $G$ and choosing the all-one gain function $\mathbf{1}$. 
Then $\Psi=(G,\mathbf{1})$ contains at least one odd cycle with gain $1$, but no odd cycles with gain $-1$.
Thus, the conclusion follows by Corollary \ref{cor: equal number of cycles sign-sym}.
For more detail regarding the case that $\re{\phi(C)}=0,$ see Section \ref{sec: allowed spec sym}.

Finally, sign-symmetry clearly implies spectral symmetry. 
\begin{lemma}
\label{lemma: sign-sym implies spec sym}
If $\Psi=(G,\psi)$ is sign-symmetric, then its spectrum is also symmetric. 
\end{lemma}
The reverse is not necessarily true, though examples to attest to this fact have historically been difficult to find.  
Some signed graph examples can be found in \cite{ghorbani2020sign}, and the remainder of this work will be dedicated to finding a much larger family of gain graphs with this property, which is ultimately found in Theorem \ref{thm: sylvester no sign-symmetry}.

\section{Spectral symmetry}\label{sec: symspec} 
In order to understand the possibility for spectral symmetry without sign-symmetry, it stands to reason that one must first have a firm grasp of the circumstances required for the former to occur. 
In particular, we investigate whether or not any graph allows a symmetric spectrum, and which graphs require it. 

\subsection{Necessary sign-symmetry}
Probably the first result that one learns about in a spectral graph theory course is the well-known theorem that a graph has a symmetric spectrum if and only if it is bipartite. 
We find a largely similar result in the context of gain graphs.
Indeed, without too much effort, one may show that all $\Psi$ on $G$ are spectrally symmetric if and only if $G$ is bipartite. 
\begin{lemma}\label{lemma: bipartite graph -> sym spec}
If $G$ is bipartite, then any $\Psi$ with $\Gamma(\Psi)\cong G$ is sign-symmetric, and thus spectrally symmetric. 
\end{lemma}

\begin{proposition}\label{prop: required iff bipartite}
The graph $G$ requires symmetric gain-spectra if and only if $G$ is bipartite.
\end{proposition}
\begin{proof}
Sufficiency follows from Lemma \ref{lemma: bipartite graph -> sym spec}. 
In order to see necessity, assume to the contrary that $G$ is not bipartite. 
Then $G$ is, itself, a gain graph (with the all-ones signature). 
Since $G$ is not bipartite, its adjacency spectrum is not symmetric, and the claim follows. 
\end{proof}

The above may be reformulated in an interesting manner. 
\begin{corollary}
The graph $G$ requires symmetric gain-spectra if and only if $G$ itself has a symmetric spectrum. 
\end{corollary}

\subsection{Allowed spectral symmetry}\label{sec: allowed spec sym}
As was clear from Lemma \ref{lemma: sym spectra coefficients} and Theorem \ref{thm: coefficients}, and even further illustrated by Proposition \ref{prop: required iff bipartite}, a central role is played by odd cycles contained in our gain graphs. 
This section gradually considers some special cases of graphs that do contain odd cycles, to conclude that any graph allows a symmetric gain-spectrum.

In case only  one odd cycle appears, it straightforwardly follows that the gain graph is spectrally symmetric if and only if said odd cycle has strictly imaginary gain. 
\begin{proposition}\label{prop: unicyclic}
Let $\Psi$ be a non-bipartite gain graph and let $\Gamma(\Psi)$ be unicyclic. 
Then $\Psi$ has a symmetric spectrum if and only if its cycle $C$ satisfies $\re{\phi(C)}=0$. 
\end{proposition}
\begin{proof}
Let $C$ be the sole cycle in $G$, whose order is $m:=2k+1$ for some $k\in\mathbb{N}$. 
It should be clear that $C$ constitutes the single order-$m$ elementary subgraph of $\Gamma(\Psi)$. 
Hence \[a_m=2\re{\phi(C)}=0\iff \re{\phi(C)}=0.\]
Furthermore, since any elementary subgraph of odd order $j>m$ must contain $C$, the contribution of said subgraph to $a_{j}$ is multiplied with $\re{\phi(C)}=0,$ hence, $a_j=0$ follows. 
Finally, note that for any odd $j<m$, there exist no elementary subgraphs of order $j$ as there are no odd cycles of order at most $j$. 
\end{proof}
Moreover, note that the above still holds up when arbitrarily many even cycles occur. 
\begin{proposition}\label{prop: single odd cycle}
Let $\Psi$ be a gain graph with arbitrarily many even-sized cycles, but only a single odd-sized cycle $C$. 
Then $\Psi$ has a symmetric spectrum if and only if $\re{\phi(C)}=0$.
\end{proposition}
An interesting distinction between the two propositions above should be mentioned here. 
Note that the graphs in Proposition \ref{prop: unicyclic} have the gains of their entire cycle space determined by the demand for spectral symmetry. 
That is, a unicyclic graph allows a spectrally symmetric gain graph, but any such gain graph belongs to the same switching equivalence class.
By contrast, while the odd cycle of the gain graphs in Proposition \ref{prop: single odd cycle} is `locked' to $\pm i$, the remaining cycles are all even and their gains are therefore of no consequence to the symmetry of the spectrum. 
Since the gain of at least one such even cycle may be freely chosen, it then follows from Theorem \ref{thm: reff switching} that infinitely many switching-distinct gain graphs with symmetric spectra occur on these underlying graphs. 
However, it should also be noted that graphs with exactly one odd cycle are  reasonably rare. 

The following can be shown analogously to Proposition \ref{prop: unicyclic}.
\begin{lemma}\label{lemma: odd cycles have gain i}
Let $\Psi$ be a gain graph whose odd-sized cycles $C$ all satisfy $\re{\phi(C)}= 0$. 
Then the spectrum of $\Psi$ is symmetric. 
\end{lemma}
Note that the reverse need not be true, as is the case in Figure \ref{fig: ugly example}. 
A straightforward application of Lemma \ref{lemma: odd cycles have gain i} leads to the conclusion that any graph is underlying to at least one equivalence class of gain graphs whose spectra are symmetric. 
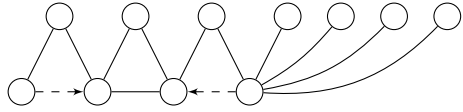
\begin{figure}[t]
\centering

    \centering
    \begin{tikzpicture}
            \node[vertex] (1) at (0,0) {};
            \node[vertex] (2) at (.5,1) {};
            \node[vertex] (3) at (1,0) {};
            \node[vertex] (4) at (1.5,1) {};
            \node[vertex] (5) at (2,0) {};
            \node[vertex] (6) at (2.5,1) {};
            \node[vertex] (7) at (3,0) {};
            
            \node[vertex] (8) at (3.5,1) {};
            \node[vertex] (9) at (4.2,1) {};
            \node[vertex] (10) at (4.9,1) {};
            \node[vertex] (11) at (5.6,1) {};
            
            \draw[negarc] (1) to node{} (3);
            \draw[edge] (1) to node{} (2);
            \draw[edge] (2) to node{} (3);
            
            \draw[edge] (3) to node{} (4);
            \draw[edge] (4) to node{} (5);
            \draw[edge] (3) to node{} (5);
            
            \draw[negarc] (7) to node{} (5);
            \draw[edge] (6) to node{} (7);
            \draw[edge] (5) to node{} (6);
            
            \draw[edge] (7) to node{} (8);
            \draw[edge,bend right=10] (7) to node{} (9);
            \draw[edge,bend right=18] (7) to node{} (10);
            \draw[edge,bend right=25] (7) to node{} (11);
    \end{tikzpicture}
    \caption{A gain graph whose spectrum is symmetric. Here, the filled edges have gain $1$ and the dashed arcs have gain $\exp(2i\pi/3)$. \cite{wissing2022spectral}}
    \label{fig: ugly example}    
\end{figure}
\begin{proposition}
Any graph $G$ allows a symmetric gain-spectrum.
\end{proposition}
\begin{proof}
Let $G$ be a graph and let $\Psi$ be the gain graph obtained from $G$ by assigning strictly imaginary gain to any edge in $E(G)$. 
Then any odd cycle $C$ in $\Psi$ has $\re{\phi(C)}= 0$, and the conclusion follows by Lemma \ref{lemma: odd cycles have gain i}. 
\end{proof}
Note, however, that this brings us no closer to finding families of gain graphs whose spectral symmetry does not imply sign-symmetry.
Indeed, any gain graph all of whose odd cycles have strictly imaginary gain is, again, necessarily sign-symmetric. 
\begin{lemma}\label{lemma: odd cycle imag -> sign symmetric. }
Let $\Psi$ be a gain graph with $\re{\phi(C)}=0$ for every odd cycle $C$. 
Then $\Psi$ is sign-symmetric. 
\end{lemma}
\begin{proof}
Observe that the gain of every even cycle is unaffected by the negation, as in the bipartite case.
Moreover, the odd cycles all have a zero real part, which is therefore also unchanged.
\end{proof}
The above should be somewhat expected to the experienced reader, as it was effectively previously observed by \citet{guo2017hermitian} in the context of Hermitian adjacency matrices. 
Nevertheless, it is a good example of the ways in which this much more general setting differs from circumstances in which the original question was posed.

\section{Constructions}\label{sec:constructions}
The logical question that follows from the above asks if we can construct spectrally symmetric graphs with at least one odd cycle whose gain has a nonzero real part. 
Clearly, the contribution of such an odd cycle must be counteracted by the contributions of other (possibly smaller) odd cycles, in such a way that their respective contributions to all odd $a_j$ are, in a sense, balanced. 
The easiest way to accomplish this would be to simply take the disjoint union of $\Psi$ and $-\Psi$. 
In order to find graphs which are connected, 
we must be a bit more resourceful. 

The idea of effectively joining $\Psi$ and $-\Psi$ together to obtain a single spectrally symmetric gain graph is quite similar to the ideas applied in various constructions of Hadamard matrices, weighing matrices and constructions of e.g. cubic graphs. 
For example, the following is a straightforward generalization of a result by \citet{ghorbani2020sign}.

\begin{lemma}[\cite{ghorbani2020sign}]\label{lemma: symspec 1}
Let $\Psi$ be an arbitrary gain graph, and denote $A=A(\Psi)$ and $B$ be a Hermitian matrix. 
Then $\tilde\Psi$, obtained from $\Psi$ as \[A(\tilde\Psi)=\begin{bmatrix} A & B\\B^* & -A\end{bmatrix},\] is sign-symmetric. In particular, it has a symmetric spectrum.
\end{lemma}
\begin{proof}
Let 
$P=\begin{bmatrix}O & I \\ I & O\end{bmatrix} \text{ and } X=\begin{bmatrix}-I & O \\ O & I\end{bmatrix}.$
Then \[PXA(\tilde\Psi)X^{-1}P^{-1}=\begin{bmatrix} -A & -B \\ -B^* & A\end{bmatrix}=-A(\tilde\Psi),\] and thus $A(\tilde{\Psi})\sim-A(\tilde\Psi).$
\end{proof}
In case $B$ is not Hermitian, the switching isomorphism above does not work in general, as is illustrated in the following example.
\example{ex: constr nonherm}{
Let $A$ be Hermitian, let $B$ be non-Hermitian with $B=C+zI$ for some Hermitian $C$ and $z\in \mathbb{T}\setminus\{\pm1\}$, and $X$ and $P$ as above. 
Then 
\[-\begin{bmatrix} A & B\\B^* & -A\end{bmatrix}=\begin{bmatrix} -A & -C-zI \\ -C^*-\bar{z}I & A\end{bmatrix}\not=\begin{bmatrix} -A & -C-\bar{z}I \\ -C^*-zI & A\end{bmatrix} =PX\begin{bmatrix} A & B\\B^* & -A\end{bmatrix}X^{-1}P^{-1}.\]
As an aside, note that if $A$ and $C$ are strictly real, then the two matrices are still switching equivalent, as the switching classes are closed under transposition, by definition.
}
However, in case the off-diagonal blocks are not Hermitian, one may still apply similar block constructions. 
The attentive reader may notice a similarity to the doubling operation used by \citet{huang2019induced} to construct signed $n$-cubes. 
This construction may indeed be used to obtain gain graphs with symmetric spectra. 

\begin{lemma}\label{lemma: symspec 2}
Let $\Psi$ be an arbitrary order-$n$ gain graph, denote $A=A(\Psi)$ and let $z\in\T$. Then $\hat{\Psi}$, defined by $A(\hat{\Psi})=\begin{bmatrix}A & zI\\\bar{z}I & -A\end{bmatrix}$, is sign-symmetric.
\end{lemma}
\begin{proof}
In the following, addition of vertex indices is modulo $2n$. 
Then, by construction, for every odd cycle $C$, whose (ordered) vertices are $v_{u_1},v_{u_2},\ldots,v_{u_k}$, the cycle $C'$ with vertices $v_{u_1+n},v_{u_2+n},\ldots,v_{u_k+n}$, is such that $\phi(C)=-\phi(C')$. 
Hence, $\Psi$ is isomorphic to $-\Psi$. 
\end{proof}

The above are all even-order gain graphs. 
Of course, one may obtain the same kind of symmetry with a central $(2n+1)$th vertex. 
To illustrate, we include the following construction. 
\begin{lemma}\label{lemma: symspec odd}
Let $\Psi$ be an arbitrary gain graph, and denote $A(\Psi)=\begin{bmatrix}0 & a\\ a^* & A'\end{bmatrix}.$
Then $\hat\Psi$, obtained from $\Psi$ as $A(\hat\Psi)=\begin{bmatrix}0 & a & -a\\ a^* & A' & O\\ -a^* & O & -A'\end{bmatrix}$ is sign-symmetric.
\end{lemma}

The above yields plenty of spectrally symmetric gain graphs. 
The main drawback of the constructions above, for the purposes of this article, is that they are also all sign-symmetric.
However, inspired by Example \ref{ex: constr nonherm}, we arrive at another construction, the foundation of which is known as Sylvester's double. 
Effectively, we more or less use a hybrid of Lemmas \ref{lemma: symspec 1} and \ref{lemma: symspec 2} to turn an arbitrary gain graph into a spectrally symmetric one. 
\begin{lemma}
\label{lemma: symspec 3}
Let $\Psi$ be an arbitrary gain graph, denote $A=A(\Psi)$ and let $z\in\mathbb{T}\cup\{0\}$.
Then $\breve{\Psi}$, defined by \begin{equation}A(\breve{\Psi})=\begin{bmatrix}A & A+zI\\ A+\bar{z}I & -A\end{bmatrix},\label{eq: main construct}\end{equation} has a symmetric spectrum.
\end{lemma}
\begin{proof}

Let $M=\begin{bmatrix}A & A+z I \\ A  + \bar{z}I & -A\end{bmatrix}.$
Then, we have
\begin{align*}\det(M-\mu I)&=\det\left(\begin{bmatrix}A-\mu I & A + zI \\ A + \bar{z}I & -A-\mu I\end{bmatrix}\right)\\&= \det((A-\mu I)(-A-\mu I)-(A+zI)(A+\bar{z}I))\\&= \det(-A^2 + \mu^2 I - (A+zI)(A+\bar{z}I)),\end{align*}
where the second equality holds since $A + \bar{z}I$ and $A-\mu I$ commute. 
Since $\mu$ appears in the final expression only as $\mu^2$, no odd-powered terms appear in the characteristic polynomial and thus the spectrum is symmetric. 
\end{proof}

The key difference of Lemma \ref{lemma: symspec 3} compared to the constructions that appeared before, is that one does not necessarily have sign-symmetry. 
An example is provided below. 
\example{ex: sylvester not sign-symmetric}{
Let $z=\exp(i\pi/3)$,

    \[A =  \begin{bmatrix}   
     0 &   1 &   i &   0 &  0 & 0 \\
     1 &   0 &   1 &  1 &   0 & 0 \\
     -i &  1 &   0 &   0 &  1 & 0 \\
     0 &  1 &   0 &   0 &   0 & 0 \\
     0 &   0 &  1 &   0 &   0 & 1 \\
     0 &   0 &  0 &   0 &   1 & 0 \end{bmatrix} \text{ and } M = \begin{bmatrix} A & A+zI\\A+\bar{z}I&-A \end{bmatrix}. \]
     Then $M$ represents a gain graph that is not sign-symmetric.}
In fact, we may abuse the fact that the off-diagonal blocks are not Hermitian to conclude that the above construction is, in general, not sign-symmetric. 
\begin{theorem}
\label{thm: sylvester no sign-symmetry}
Let $\Psi$ be a connected non-bipartite gain graph with $\Gamma(\Psi)$ asymmetric, denote $A=A(\Psi)$ and let $z\in\mathbb{T}\setminus\{\pm 1, \pm i\}$. 
Then $\breve{\Psi}$, defined by $M:=A(\breve{\Psi})=\begin{bmatrix}A & A+zI\\ A+\bar{z}I & -A\end{bmatrix}$, is sign-symmetric if and only if $\Psi$ is switching equivalent to a signed graph.
\end{theorem}
\begin{proof}
Suppose that $\Psi$ is switching equivalent to a signed graph. Then, w.l.o.g., it may switched such that the entries of $A(\Psi)$ are all real. Let $P=\begin{bmatrix}O&I\\I&O\end{bmatrix}$ and $X=\begin{bmatrix}-I&O\\O&I\end{bmatrix}$. 
Then 
\begin{equation}
    \label{eq: signed graph case}
    PX M X^{-1}P^{-1} = \begin{bmatrix}-A & -A-\bar{z}I\\ -A-zI & A\end{bmatrix} = -M^\top. 
\end{equation}
Since switching isomorphism classes are by definition closed under taking the converse, we have that $-M^{\top}\sim -M$, and hence \eqref{eq: signed graph case} shows that $M\sim -M.$
Note that taking the transpose only affects the diagonal entries of the off-diagonal blocks, since $A$ is strictly real and thus symmetric. 

We consider the reverse implication. 
By construction,  \[\Gamma(M)=\begin{bmatrix}|A| & |A|+I\\|A|+I & |A|\end{bmatrix}\] (where $|A|$ is the entrywise absolute value of $A$) has precisely the full automorphism group generated by all of the transpositions $(v,v+n)$, with $v\in[n]$, where $n$ is the order of $\Psi$. 
Since a switching isomorphism from $M$ to $-M$ is an automorphism of $\Gamma(M)$, it therefore follows that the permutation matrix $P$, representing the relabeling of the vertices, is symmetric. That is, $P^2=I$. 

Note that since $\Psi$ is non-bipartite, there is at least an odd cycle $C'$ in $\breve{\Psi}$ whose gain is affected by the negation, i.e., $\re{\phi(C')}\not=0$.
Indeed, if all odd cycles in $\Psi$ would be strictly imaginary; one could take an odd cycle $C=(v_{c_1},v_{c_2},\ldots,v_{c_k})$ from $\Psi$ and add two vertices from the negated copy to obtain $C'=(v_{c_1},v_{c_1+n},v_{c_2},v_{c_2+n},v_{c_3},v_{c_4},\ldots,v_{c_k})$. 
Then $\phi(C')=z^2\phi(C),$ and thus $\re{\phi(C')}\not=0,$ showing our claim.
Since switching does not affect the gain of a cycle, $M\sim-M$ now implies $P\not=I$. 

Assume, for now, that we do not take the converse. 
Then there are some matrices $P,X$ such that $PXMX^{-1}P^{-1}= -M$, where $P$ is a (non-identity) symmetric permutation matrix and $X$ is a diagonal matrix whose diagonal entries are complex units. 
Let $\mathcal{R}\subseteq V$ denote the vertices that are not fixed by $P$. 
Observe that $|\mathcal{R}|=0$ implies $P=I$, which contradicts the above. 
Assume without loss of generality that $\mathcal{R}=\{1,\ldots,r\}$ for some $1\leq r< n$ (the case $r=n$ is treated later), and decompose $A$ as 
\[A=\begin{bmatrix} A_R & B \\ B^* & A_F\end{bmatrix},\]
where $A_R\in\mathbb{C}^{r\times r}$ and the remaining blocks are appropriately dimensioned.
Accordingly decompose $X$ as $X=diag(X_R,X_F,Y_R,Y_F),$ where $X_R,Y_R\in\mathbb{C}^{r\times r}$ and $X_F,Y_F\in\mathbb{C}^{n-r\times n-r}$. 
Then
\begin{align}\label{eq: M no transposition}
    &M':=PXA(\breve\Psi)X^{-1}P^{-1}=\\&
    \label{eq: decomposed and switched}\begin{bmatrix}
    -Y_R A_R Y_R^{-1}             & Y_R B X_F^{-1}                & Y_R (A_R+\bar{z}I) X_R^{-1}    & -Y_R B Y_F^{-1}         \\
    X_F B^* Y_R^{-1}            & X_F A_F X_F^{-1}                & X_F B^* X_R^{-1}               & X_F (A_F+zI) Y_F^{-1}         \\
    X_R (A_R+zI) Y_R^{-1}         & X_R B X_F^{-1}                & X_R A_R X_R^{-1}               & X_R B Y_F^{-1}         \\
    -Y_F B^* Y_R^{-1}           & Y_F (A_F+\bar{z}I) X_F^{-1}    & Y_F B^* X_R^{-1}             & -Y_F A_F Y_F^{-1}         
    \end{bmatrix},\nonumber
\end{align}
which equals $-M$ on all entries. 
Observe from the (1,3) block that $\bar{z}Y_RX_R^{-1}=-zI$, and thus $Y_R=-z^2X_R$. 
Furthermore, from blocks (1,2) and (3,2), we have $Y_RBX_F^{-1}=-B=X_RBX_F^{-1}$. 
Now, since $B$ is non-zero (otherwise $\Psi$ is not connected), these two equations jointly imply $z^2=-1,$ which is a contradiction. 

What remains is to consider $r=n$. 
Then similarly to before
\begin{equation}\label{eq: M when r = n}
M'=\begin{bmatrix}
    -X_R A_R X_R^{-1}                        & -z^2X_R (A_R+\bar{z}I) X_R^{-1}       \\
   -\bar{z}^2 X_R (A_R+zI) X_R^{-1}                     & X_R A_R X_R^{-1}              
    \end{bmatrix},
    \end{equation}
and thus $-X_RA_RX_R^{-1}=-A_R$ and $-z^2X_RA_RX_R^{-1}=-A_R$, which in this case implies $z^2=1$  since $A_R$  is non-zero. This is again a contradiction. 

Finally, we revisit the matter of taking the converse. 
Let $1\leq r<n$ and, contrary to before, suppose now that $M'=-M^\top,$ where $M'$ is as in \eqref{eq: M no transposition}. 
Then $\bar{z}Y_RX_R^{-1}=-\bar{z}I$ and thus $X_R=-Y_R$. 
Note the similarity to before; a parallel argument may be formulated with little effort. 

In case $r=n$, some extra effort is needed. Plugging in $X_R=-Y_R$ into \eqref{eq: M no transposition} and equating the result to $-M^\top$ yields $X_RA_RX_R^{-1}=A_R^\top,$ which is equivalent to \begin{equation}\label{eq: proof 4.5 final}[X_R]_{uu}A_{uv}[X_R]_{vv}^{-1}=\overline{A_{uv}}~~~\forall(u,v)\in E.\end{equation} 
Note that if $A_{uv}=1$, then it clearly follows that $[X_R]_{uu}=[X_R]_{vv}.$
Now, since one may without loss of generality assume (see \cite[Cor. 3.6.1]{wissing2022spectral}) that a spanning tree of the edges has gain $1$, it follows that $X_R=cI$ for some $c\in\T$. 
But then \eqref{eq: proof 4.5 final} reduces to $A_{uv}=\overline{A_{uv}},$ which implies $A\in\mathbb{R}^{n\times n}$ and $\Psi$ is therefore switching equivalent to a signed graph. 
\end{proof}

Using Lemma \ref{lemma: symspec 3} and Theorem \ref{thm: sylvester no sign-symmetry}, we thus obtain a huge family of gain graphs that are spectrally symmetric but not sign-symmetric. 
In particular, since there is substantial freedom in the choice of $\Psi$, there should be examples of such gain graphs with many graph theoretical properties. 
In light of the comparatively narrow families of signed graphs presented by \citet{ghorbani2020sign}, this is quite surprising. 

\normalsize

\section*{Acknowledgements}
     We thank Pieter Kleer and Peter de Groot for their contributions to this paper.

\scriptsize
\bibliography{mybib}{}
\bibliographystyle{plainnat}
\normalsize

\end{document}